\documentclass[reqno,12pt,letterpaper]{amsart}
\usepackage{amsmath,amssymb,amsthm,graphicx,mathrsfs,url,bbm}
\usepackage[usenames,dvipsnames]{color}
\usepackage[colorlinks=true,linkcolor=Red,citecolor=Green]{hyperref}

\def\?[#1]{\textbf{[#1]}\marginpar{\Large{\textbf{??}}}}
\newtheorem{thm}{Theorem}
\newtheorem{prop}{Proposition}

\newtheorem{lem}[prop]{Lemma}

\newtheorem{rem}[prop]{Remark}

\numberwithin{equation}{section}
\numberwithin{prop}{section}

\renewcommand{\Im}{\mathop{\rm Im}\nolimits}

\DeclareMathOperator{\Op}{Op}

\DeclareMathOperator{\supp}{supp}

\DeclareMathOperator{\nbd}{nbd}
\DeclareMathOperator{\el}{ell}

\DeclareMathOperator{\WF}{WF}
\DeclareMathOperator{\Res}{{\rm Res}}

\newcommand{\Fcal}{{\mathcal F}}

\newcommand{\Hcal}{{\mathcal H}}

\newcommand{\Ocal}{{\mathcal O}}

\newcommand{\Scal}{{\mathcal S}}
\newcommand{\RR}{{\mathbb R}}

\begin{document}
\title{Flat trace estimates for Anosov flows}

\author{Long Jin}
\email{jinlong@mail.tsinghua.edu.cn}
\address{Yau Mathematical Sciences Center, Tsinghua University, Beijing, China \& Beijing Institute of Mathematical Sciences and Applications, Beijing, China}

\author{Zhongkai Tao}
\email{ztao@math.berkeley.edu}
\address{Department of Mathematics, Evans Hall, University of California,
Berkeley, CA 94720, USA}

\begin{abstract}
We prove a high energy flat trace estimate for the modified resolvent of the generator of an Anosov flow. This fills a gap in the proof of the local trace formula in \cite{local} and is a by-product of the authors' ongoing project of its generalization to Axiom A flows.
\end{abstract}

\maketitle
\section{Introduction}
\label{s:intro}
This note is a by-product of the authors' ongoing project on the local trace formula for Axiom A flows, which leads to the discovery of some issues in \cite{local}. Since the situation for Anosov flows is simpler than the one for Axiom A flows, we give here a separate presentation to fix the issues in \cite{local}. 

Let $X$ be a smooth compact manifold, $\varphi_t:X\to X$ be an Anosov flow generated by a smooth vector field $V$, and $P=-iV$, Jin--Zworski \cite{local} proved the following local trace formula relating the Pollicott--Ruelle resonances $\Res(P)$ to the lengths of closed geodesics.
\begin{thm}\label{tracef}
For any $A>0$ there exists a distribution $F_A\in \Scal'(\RR)$ supported in $[0,\infty)$ such that
\begin{align*}
    \sum\limits_{\mu\in\Res(P),\Im(\mu)>-A}e^{-i\mu t}+F_A(t)=\sum\limits_{\gamma }\frac{T_\gamma^\#\delta(t-T_\gamma)}{|\det(I-\mathcal{P}_\gamma)|},\quad t>0
\end{align*}
in $\mathcal{D}'((0,\infty))$, where the sum on the right hand side is taken over all closed geodesics, $\mathcal{P}_\gamma$ is the Poincar\'e map, and
\begin{align}
\label{e:error}
    |\widehat{F}_A(\lambda)|=\Ocal_{A,\varepsilon}(\langle \lambda\rangle^{2n+1+\varepsilon}),\quad \Im \lambda<A-\varepsilon
\end{align}
for any $\varepsilon>0$.
\end{thm}

The last estimate \eqref{e:error} has been modified comparing to \cite[(1.5)]{local}. The additional loss of $\varepsilon$ in the exponent in \eqref{e:error} comes from the following mistake in \cite{local}: rescaling from \cite[(4.20)]{local} back to \cite[(4.1)]{local}, we should gain an additional $h$ from the derivative changing from $\frac{d}{dz}$ to $\frac{d}{d\lambda}$, but also have $|z|=h|\lambda|\sim h^{1/2}$ and thus the result should be $h^{-2n}\sim\lambda^{4n}$. However, we can go back to the setting of \cite[Proposition 3.4]{zeta} and replace $h^{1/2}$ by any $h^{\varepsilon}$ with $\varepsilon\in(0,1)$ arbitrarily small. This way we also replace the bound in \cite[(4.19)]{local} and \cite[(4.20)]{local} by $h^{-(2-\varepsilon)n-2}$ and thus we obtain the bound in \eqref{e:error}. In section \ref{s:flattrace}, we will give a simpler proof for a weaker high energy flat trace estimate, comparing to \cite[Proposition 3.1]{local}, see Theorem \ref{traceest}. From this, the bound in \cite[(4.20)]{local} becomes $h^{-2n-2}$, but still gives the same bound in \eqref{e:error}. The advantage is that we can avoid the complicated construction for complex absorbing potential $Q$ as in \cite[\S 2.5]{local}.

In \cite{local}, the proof for the high energy flat trace estimate \cite[Proposition 3.1]{local} was incomplete as it relied on the following flawed statement (\cite[(2.14)]{local}) about the semiclassical wavefront set for the resolvent $R_h(z)=(hP-z)^{-1}$:
\begin{align*}
    \WF_h'(R_h(z))\cap S^*(X\times X)\subset \kappa(\Delta(T^*X)\cup \Omega_+\cup(E_u^*\times E_s^*)\setminus\{0\}),
\end{align*}
which was used to deduce the same statement \cite[(2.19)]{local} for the modified resolvent $\widetilde{R}_h(z)=(hP-iQ-z)^{-1}$. However, $R_h(z)$ has poles which are exactly the Pollicott--Ruelle resonances. Even in the set where it is well-defined, it is not clear that the kernel is $h$-tempered uniformly in $z$, and thus $\WF_h'(R_h(z))$ may not be defined. To remedy this issue, we analyze the modified resolvent $\widetilde{R}_h(z)$ directly to give the statement \cite[(2.19)]{local}, which is the correct statement eventually used in the proof of Theorem \ref{tracef} in \cite{local}. This will be done in Proposition \ref{wf} in Section \ref{s:wavefront}.

For more details on the notations we refer to \cite{local}. For preliminaries on semiclassical analysis we refer to Zworski \cite{semibook} and Dyatlov--Zworski \cite[Appendix E]{resbook}. For other recent developments concerning trace formulas for Pollicott-Ruelle resonances, see \cite{Je1}, \cite{Je2}.

\section{Wavefront set estimates}
\label{s:wavefront}

In this section, we fix the issue in \cite{local} by proving the following semiclassical wavefront set estimate for the modified resolvent $\widetilde{R}_h(z)$. We briefly recall the notations from \cite{local}: Let $Q$ be the absorbing potential as in \cite{local}, to be more precise, we require
\begin{itemize}
\item $\WF_h(Q)\subset\{|\xi|<1\}$;
\item $\sigma_h(Q)>0$ on $\{|\xi|\leq 1/2\}$;
\item and $\sigma_h(Q)\geq0$ everywhere.
\end{itemize}
The additional requirement in \cite[\S 2.5]{local} is used to improve the power in the flat trace estimate \eqref{crucial} and we will give a simpler argument in Secion \ref{s:flattrace} to avoid the complications. In \cite[Proposition 3.4]{zeta}, it is shown that for fixed $C_1,C_2,\varepsilon>0$, $\widetilde{P}_h(z)=hP-iQ-z$ is invertible for $z\in [-C_1h^\varepsilon,C_1h^\varepsilon]+i[-C_2h,1]$ and its inverse satisfied the following estimate
\begin{align*}
    \|\widetilde{R}_h(z)\|_{\Hcal^s_h\to \Hcal^s_h}\leq Ch^{-1}.
\end{align*}
Here $\Hcal^s_h=H_{sG(h)}$ is the semiclassical anisotropic Sobolev space defined in \cite[\S 3.3]{zeta} and $s>0$ is a parameter chosen large enough depending on $C_1$ and $C_2$. The weight function $G(h)$ is constructed in a way that $\widetilde{P}_h(z):D^s_h:=\{u\in \Hcal^s_h: \widetilde{P}_h(z)u\in\Hcal^s_h\}\to \Hcal^s_h$ is invertible. In the following we will only use the fact that 
\begin{align*}
    H^s_h\subset \Hcal^s_h\subset H^{-s}_h,
\end{align*}
where $H^s_h$ is the usual semiclassical Sobolev spaces on $X$.

\begin{prop}
\label{wf}
We have 
\begin{align}
\label{e:semiwf}
    \WF_h'(\widetilde{R}_h(z))\cap S^*(X\times X)\subset \kappa(\Delta(T^*X)\cup \Omega_+\cup(E_u^*\times E_s^*)\setminus\{0\})
\end{align}
where $\Omega_+$ is the flowout
\begin{align*}
    \Omega_+=\{(e^{tH_p}(y,\eta),y,\eta)\,:\, p(y,\eta)=0\}\subset T^\ast(X\times X)\simeq T^\ast X\times T^\ast X,
\end{align*}
and $\kappa:T^\ast(X\times X)\setminus\{0\}\to S^\ast(X\times X)$ is the natural projection map.
\end{prop}

\begin{rem}
Note that $S^*(X\times X)\neq S^*X\times S^*X$, hence there are difficulties to deal with the fiber infinity directly. In fact, unlike the finite part of the wavefront set $T^\ast(X\times X)\simeq T^\ast X\times T^\ast X$, there is no natural way to identify the element in $S^* X\times S^*X$ where $S^\ast X=\kappa(T^\ast X\setminus\{0\})$ with the element in $S^*(X\times X)=\kappa(T^\ast(X\times X)\setminus\{0\})$. However, we do have the natural identification of the diagonal elements $\Delta(S^\ast X)=\kappa(\Delta(T^\ast X)\setminus\{0\})$.
\end{rem}

The rest of this section will be devoted to the proof of Proposition \ref{wf}.
We will follow the strategy of \cite[Proposition 3.4]{zeta}, where the authors prove the estimate for the finite part of $\WF_h'(\widetilde{R}_h(z))$. To deal with the wavefront set at fiber infinity we introduce another small parameter $\tilde{h}>0$ (which will play the role of $|(\xi,\eta)|^{-1}$).

\noindent\textbf{Step 1:} Let $p^{-1}(0)=\{(x,\xi)\in T^*X: p(x,\xi)=0\}\supset E_u^\ast\cup E_s^\ast$, we first show a weaker statement:
\begin{align}
\label{e:semiwfweak}
    \WF_h'(\widetilde{R}_h(z))\cap S^*(X\times X)\subset \kappa(\Delta(T^*X)\cup \Omega_+\cup(E_u^*\times p^{-1}(0))\setminus\{0\}).
\end{align}
Suppose $(x_0,\xi_0,y_0,\eta_0)\in\{|(\xi,\eta)|= 1\}\setminus( \Delta(T^*X)\cup \Omega_+\cup (E_u^*\times p^{-1}(0)))$, then as in \cite[Proposition 3.4]{zeta} there exist $\rho>0$ and neighbourhoods $U$ of $(x_0,\rho\xi_0)$ and $W$ of $(y_0,\rho\eta_0)$, and $A,B\in\Psi^0_{h}(X)$ such that
\begin{align}\label{pos}
    \begin{split}
    \|Au\|_{\mathcal{H}^s_{h}}\leq Ch^{-1}\|B\widetilde{P}_h(z)u\|_{\mathcal{H}^s_{h}}+\mathcal{O}(h^\infty)\|u\|_{H_{h}^{-N}},\\
    U\subset\el_{h}(A),\quad (\{|\xi|\leq 1\}\cup W)\cap \WF_{h}(B)=\varnothing.
    \end{split}
\end{align}
Moreover, $A$ is microlocally supported near $(x_0,\rho\xi_0)$ and $B$ microlocally supported in a neighbourhood of $\{e^{-tH_p}(x_0,\rho\xi_0): t\geq 0\}$. The condition that $(x_0,\xi_0,y_0,\eta_0)\notin E_u^*\times p^{-1}(0)$ guarantees that $\WF_h(B)\cap \{|\xi|\leq 1\}=\varnothing$ for some large number $\rho>0$. We can also assume that 
\begin{align*}
   A=\Op_h(a),\quad B=\Op_h(b),\quad Q=\Op_h(q)
\end{align*}
with symbols $b\in S^0$ and $a,q\in C_0^\infty$ independent of $h$, and $\supp q\subset \{|\xi|\leq 1\}$ so that $\supp q\cap \supp b=\varnothing$.
Here $\Op_h$ denotes a semiclassical quantization on a compact manifold, see \cite[Appendix E]{resbook}.

Replacing $h$ by $h\tilde{h}$ in the estimate \eqref{pos}, we get 
$$A_{\tilde{h}}=\Op_{h\tilde{h}}(a),\quad B_{\tilde{h}}=\Op_{h\tilde{h}}(b),\quad Q_{\tilde{h}}=\Op_{h\tilde{h}}(q)\in\Psi^0_{h\tilde{h}}(X)$$
such that
\begin{align}
\label{postilde}
\begin{split}
        \|A_{\tilde{h}}u\|_{\mathcal{H}^s_{h\tilde{h}}}\leq& C(h\tilde{h})^{-1}\|B_{\tilde{h}}(h\tilde{h}P-\tilde{h}z-iQ_{\tilde{h}} )u\|_{\mathcal{H}^s_{h\tilde{h}}}+\mathcal{O}((h\tilde{h})^\infty)\|u\|_{H_{h\tilde{h}}^{-N}},\\
    &\begin{gathered}U\subset\el_{h\Tilde{h}}(A_{\tilde{h}}),\quad (\{|\xi|\leq 1\}\cup W)\cap \WF_{h\Tilde{h}}(B_{\tilde{h}})=\varnothing.\end{gathered}
\end{split}
\end{align}
Note $z\in [-C_1h^\varepsilon,C_1h^\varepsilon]+i[-C_2h,1]$ implies $\tilde{h}z\in [-C_1 (\tilde{h}h)^\varepsilon,C_1(\tilde{h}h)^\varepsilon]+i[-C_2\tilde{h}h,1]$. However we wish to recover $\widetilde{P}_h$ in estimate \eqref{postilde}, and this require us to replace $Q_{\tilde{h}}$ by $\tilde{h}Q$ and to deal with the $Q$ term. We need the following lemma:
\begin{lem}
For every $N\in\mathbb{N}$,
\begin{align*}
    \|B_{\tilde{h}}Qu\|_{H^N_{h\tilde{h}}}=\mathcal{O}(h^\infty \tilde{h}^\infty)\|u\|_{H^{-N}_{h\tilde{h}}}.
\end{align*}
\end{lem}
\begin{proof}
Using a partition of unity argument we may assume that we are on $\mathbb{R}^n$ and all the symbols are compactly supported in $\mathbb{R}^n$.
Recall (e.g. \cite[Theorem 4.23]{semibook}) for a sufficiently large constant $M>0$,
\begin{align*}
    \|\Op_h(a)\|_{L^2\to L^2}\lesssim \|a\|_{S^{0,M}},\quad \|a\|_{S^{k,M}}:=\sum\limits_{|\alpha|+|\beta|\leq M}\left\|\langle \xi\rangle^{|\alpha|-k}\partial^\beta_x\partial^\alpha_{\xi}a(x,\xi)\right\|_{L^\infty}.
\end{align*}
Since $\{\xi=0\}\cap\supp b=\varnothing$, for $m\gg 1$,
\begin{align*}
    \|B_{\tilde{h}}Q\|_{H^{-N}_{h\tilde{h}}\to H^N_{h\tilde{h}}}&=\|\langle h\tilde{h}D\rangle^N B_{\tilde{h}}Q\langle h\tilde{h}D\rangle^N\|_{L^2\to L^2}\\
    &\lesssim \|\langle \tilde{h}\xi\rangle^N\# b(x,\tilde{h}\xi)\# q(x,\xi)\# \langle \tilde{h}\xi\rangle^N\|_{S^{0,M}}\\
    &\lesssim h^m\|\langle \tilde{h}\xi\rangle^N\|^2_{S^{N,M'}} \|b(x,\tilde{h}\xi)\|_{S^{m,M'}} \|q(x,\xi)\|_{S^{-m-2N,M'}}\\
    &\lesssim \Ocal(h^m \tilde{h}^m).
\end{align*}
Since $m$ can be chosen arbitrarily large, this concludes the proof.
\end{proof}

Now we go back to \eqref{postilde} and taking $u(x)=\widetilde{R}_h(z)(\psi(x)e^{ix\cdot \rho\eta_0/h\tilde{h}})$ (here we choose a local coordinates and identify a neighborhood of $x_0$ to subset of $\mathbb{R}^n$) where $\supp\psi\times \{\rho\eta_0\}\subset W$, the wavefront set condition \eqref{pos} for $B$ gives
\begin{align*}
    \|B_{\tilde{h}}Q_{\tilde{h}}\|_{H^{-N}_{h\tilde{h}}\to H^N_{h\tilde{h}}} =\Ocal(h^\infty\tilde{h}^\infty),\quad \|B_{\tilde{h}}(\psi(x)e^{ix\cdot \rho\eta_0/h\Tilde{h}})\|_{H^N_{h\tilde{h}}}=\Ocal(h^\infty\tilde{h}^\infty).
\end{align*}
Therefore we have
\begin{align*}
    \|A_{\tilde{h}}u\|_{\mathcal{H}^r_{h\tilde{h}}}&\leq Ch^{-1}\|B_{\tilde{h}}\widetilde{P}_hu\|_{\mathcal{H}^r_{h\tilde{h}}}+C(h\tilde{h})^{-1}\|B_{\tilde{h}}Q_{\tilde{h}} u\|_{\mathcal{H}^r_{h\tilde{h}}}\\
    &\quad\quad +Ch^{-1}\|B_{\tilde{h}}Q u\|_{\mathcal{H}^r_{h\tilde{h}}}+\Ocal((h\tilde{h})^\infty)\|u\|_{H_{h\tilde{h}}^{-N}}
    \\
    &=\mathcal{O}(h^{-1})\|B_{\tilde{h}}(\psi(x)e^{ix\cdot r\eta_0/h\Tilde{h}})\|_{\mathcal{H}^r_{h\tilde{h}}}+\mathcal{O}(h^\infty\tilde{h}^\infty)\|u\|_{H_{h\tilde{h}}^{-N}}\\
    &=\mathcal{O}(h^\infty\Tilde{h}^\infty).
 \end{align*}
This means $\WF_{h\Tilde{h}}(u)\cap U=\varnothing$, and thus if $\chi\in C^\infty(X)$ and $\supp \chi \times \{\rho\xi_0\}\subset U$, then
\begin{align*}
        \int\chi(x)e^{-ix\cdot \rho\xi_0/h\Tilde{h}}\widetilde{R}_h(z)(\psi(x)e^{ix\cdot \rho\eta_0/h\Tilde{h}})dx=\mathcal{O}(h^\infty\Tilde{h}^\infty)
\end{align*}
Moreover, by construction it is easy to see the estimate is locally uniform in $(x_0,\xi_0,y_0,\eta_0)$. Therefore by the equivalent definition of semiclassical wavefront sets using the semiclassical Fourier transform (see \cite[Definition 3.2]{wavefront}), $\kappa(x_0,\xi_0,y_0,\eta_0)=\kappa(x_0,\rho\xi_0,y_0,\rho\eta_0)\not\in\WF_h'(\widetilde{R}_h(z))\cap S^\ast(X\times X)$ and we have \eqref{e:semiwfweak}.

\noindent\textbf{Step 2:} The previous method does not work for $(x_0,\xi_0,y_0,\eta_0)\in E_u^*\times p^{-1}(0)$ since $\WF_h(B)$ has to intersect the zero section $\{\xi=0\}$. Here we argue by duality.  Suppose $(x_0,\xi_0,y_0,\eta_0)\in  \{|(\xi,\eta)|= 1\}\setminus(\Delta\cup \Omega_+\cup(p^{-1}(0)\times E_s^*))$, we consider the following operator
\begin{align*}
    -\widetilde{P}_h(z)^*:=-hP-iQ-(-\bar{z}),
\end{align*}
acting on $\Hcal^{-s}_h$. We see that this corresponds to the reversed Anosov flow $\varphi_{-t}$ generated by $-V$ and $z\in[-C_hh^\varepsilon,C_1h^\varepsilon]+i[-C_2h,1]$ also gives $-\bar{z}$ in the same region. We can repeat the same argument with the opposite propagation direction we get $\widetilde{P}_h(z)^*$ is invertible, with inverse $\widetilde{R}_h(z)^*:\Hcal^{-s}_h\to\Hcal^{-s}_h$ satisfying
    \begin{align*}
        \|\widetilde{R}_h(z)^*\|_{\Hcal^{-s}_h\to\Hcal^{-s}_h}\leq Ch^{-1}.
    \end{align*}
Moreover, there exist $\rho>0$, $U=\nbd(x_0,\rho\xi_0)$ and $W=\nbd(y_0,\rho\eta_0)$ such that for $\supp\psi\times\{\rho\eta_0\}\subset W$ and $\supp\chi\times\{\rho\xi_0\}\subset U$ we have
    \begin{align*}
        \int\psi(x)e^{-ix\cdot\rho\eta_0/h\Tilde{h}}\widetilde{R}_h(z)^*(\chi(x)e^{ix\cdot\rho\xi_0/h\Tilde{h}})dx=\mathcal{O}(h^\infty\Tilde{h}^\infty),
    \end{align*}
and the estimate is locally uniform in $(x_0,\xi_0,y_0,\eta_0)$. Therefore $\kappa(y_0,\eta_0,x_0,\xi_0)\not\in \WF_h'(\widetilde{R}_h(z)^\ast)\cap S^\ast(X\times X)$. Since the Schwartz kernel of $\widetilde{R}_h(z)^\ast$ is $\overline{K(y,x)}$ if the $K(x,y)$ is the Schwartz kernel of $\widetilde{R}_h(z)$, we have $\kappa(x_0,\xi_0, y_0,\eta_0)\not\in \WF_h'(\widetilde{R}_h(z))\cap S^\ast(X\times X)$ and thus
$$\WF_h'(\widetilde{R}_h(z))\cap S^\ast(X\times X)\subset\kappa(\Delta(T^\ast X)\cup\Omega_+\cup(p^{-1}(0)\times E_s^\ast))\setminus\{0\}).$$
Combining this with \eqref{e:semiwfweak} we get the desired estimate \eqref{e:semiwf} and finish the proof of Proposition \ref{wf}.

\section{Flat trace estimates}
\label{s:flattrace}
In this section, we present a simpler argument than the one in \cite{local} to give the following flat trace estimate (see \cite[Proposition 3.1]{local}). The result is slightly weaker than the original one in \cite{local}, but avoid using \cite[Proposition 10.3]{nozw} and thus the assumption \cite[(2.7)]{local} for the complex absorbing potential $Q$.


\begin{thm}
\label{traceest}
The flat trace
\begin{align*}
    T(z)={\rm tr}^\flat ( e^{-it_0h^{-1}\widetilde{P}_h(z)} \widetilde{R}_h(z))
\end{align*}
is well-defined and holomorphic for $z$ in $[-C_1h^\varepsilon,C_1h^\varepsilon]+i[-C_2h,1]$. Moreover, we have 
\begin{align}\label{crucial}
    T(z)=\mathcal{O}(h^{-2n-2}).
\end{align}
\end{thm}

To prove it we need a wavefront set estimate for the Schwartz kernel of $e^{-it_0h^{-1}\widetilde{P}_h(z)} \widetilde{R}_h(z)$:
\begin{lem}
\label{wf2}
\begin{align*}
{\rm WF}_h'(e^{-it_0h^{-1}\widetilde{P}_h(z)} \widetilde{R}_h(z))\cap S^*(X\times X)&\subset  \\
\kappa(\{(x,\xi,y,\eta)\,:\,(e^{-t_0H_p}(x,\xi),y,\eta)&\in \Delta(T^*X)\cup \Omega_+\cup(E^*_u\times E^*_s)\setminus\{0\} \text{ or } \xi=0,\eta\neq 0\}).
\end{align*}
\end{lem}
\begin{proof}
Proposition \ref{wf} gives
\begin{align*}
    \WF_h'(\widetilde{R}_h(z))\cap S^*(X\times X)\subset \kappa(\Delta(T^*X)\cup \Omega_+\cup(E_u^*\times E_s^*)\setminus\{0\}).
\end{align*}
Thus
\begin{align*}
     {\rm WF}_h'(e^{-t_0V}\widetilde{R}_h(z))\cap S^*(X\times X)\subset \\
\kappa(\{(x,\xi,y,\eta)\,:\,(e^{-t_0H_p}(x,\xi),y,\eta)&\in \Delta(T^*X)\cup \Omega_+\cup(E^*_u\times E^*_s)\}\setminus\{0\}).
\end{align*}
We have
\begin{align*}
    e^{-it_0P}- e^{-it_0h^{-1}(hP-iQ)}=h^{-1}\int_0^{t_0}e^{-i(t_0-t)P}Qe^{-ith^{-1}(hP-iQ)}dt,
\end{align*}
and using $\WF_h'(Q)\cap S^\ast(X\times X)=\emptyset$ and \cite[Lemma 3.7(iii)]{wavefront}, we can compute 
\begin{align*}
    \WF_h'(e^{-i(t_0-t)P}Qe^{-ith^{-1}(hP-iQ)}\widetilde{R}_h(z))\cap S^*(X\times X)\subset(X\times\{0\})\times S^*X.
\end{align*}
Therefore
\begin{align*}
    &\WF_h'(e^{-it_0h^{-1}\widetilde{P}_h(z)}\widetilde{R}_h(z))\cap S^*(X\times X)\\
    &\subset  \left({\rm WF}_h'(e^{-it_0P}\widetilde{R}_h(z))\bigcup \cup_{t=0}^{t_0}\WF_h'(e^{-i(t_0-t)P}Qe^{-ith^{-1}(hP-iQ)}\widetilde{R}_h(z))\right)\bigcap S^*(X\times X)\\
    &\subset \kappa(\{(x,\xi,y,\eta)\,:\,(e^{-t_0H_p}(x,\xi),y,\eta)\in \Delta(T^*X)\cup \Omega_+\cup(E^*_u\times E^*_s)\setminus\{0\} \text{ or } \xi=0,\eta\neq 0\}).
\end{align*}
\end{proof}

Theorem $\ref{traceest}$ then follows from the following general lemma.
\begin{lem}\label{traceest_lem}
Let $X$ be an $n$-dimensional smooth manifold and $m\in \mathbb{R}$. If $P(h):C^\infty(X)\to \mathcal{D}'(X)$ is $h$-tempered and satisfies
\begin{itemize}
    \item ${\rm WF}_h'(P(h))\cap \Delta(S^*X)=\varnothing$;
    \item $\|AP(h)B\|_{L^2\to L^2}=\mathcal{O}(h^{-m})$ for $A,B\in \Psi^{\rm comp}_h(X)$;
\end{itemize}
then ${\rm tr}^\flat (P(h))$ is well-defined with
\begin{align*}
    {\rm tr}^\flat (P(h))=\mathcal{O}(h^{-2n-m}).
\end{align*}
\end{lem}
\begin{proof}
Since ${\rm WF}_h'(P(h))\cap \Delta(S^*X)=\varnothing$, we have ${\rm WF}'(P(h))\cap \Delta(T^*X)=\varnothing$,
it is then a classical theorem (see e.g. \cite[Theorem 8.2.4]{micro1}) that the flat trace is well-defined as long as the wavefront set does not intersect the diagonal. 

Let $u=K_h$ be the Schwartz kernel of $P(h)$, $\iota:X\to X\times X$ be the diagonal embedding, then for $\chi\in C^\infty(X)$, $\varphi(x,y)=\psi(x)\psi(y)\in C^\infty(X\times X)$ supported near the diagonal, 
\begin{align}\label{fouriertrace}
    \langle\iota^*(\varphi u), \chi \rangle=\langle \varphi u,\iota_*\chi\rangle=\frac{1}{(2\pi h)^{2n}}\int\Fcal_h(\varphi u)I_{\chi,h}(\xi,\eta)d\xi d\eta
\end{align}
where 
\begin{align*}
    I_{\chi,h}(\xi,\eta)=\int\chi(x)e^{ix\cdot(\xi+\eta)/h}dx.
\end{align*}
If $|\xi+\eta|>|\xi|/C$, then
\begin{align*}
    I_{\chi,h}(\xi,\eta)=\Ocal(h^\infty(|\xi|+|\eta|)^{-\infty}).
\end{align*}
Thus we only need to consider the case when $(\xi,\eta)$ lies in a small conical neighbourhood of $\{\xi+\eta=0\}$ or in a neighbourhood of $\{\xi=\eta=0\}$.

\begin{itemize}
    \item[(i)] When $|\xi|+|\eta|\leq C$ is bounded, we have for some $A,B\in\Psi^{\rm comp}_h(X)$
    \begin{align*}
        |\Fcal_h(\varphi u)|&=|\langle P(h)B(\psi(y)e^{-iy\cdot\eta/h}),  A(\psi(x)e^{-ix\cdot\xi/h})\rangle|+\Ocal(h^\infty)\\
        &\lesssim \|AP(h)B\|_{L^2\to L^2}+\Ocal(h^\infty)\\
        &=\Ocal(h^{-m}).
    \end{align*}
    \item[(ii)] When $(\xi,\eta)$ is near fiber infinity and in a small conic neighbourhood of $\{\xi+\eta=0\}$ which does not intersect $\WF_h'(P(h))$, we have
    \begin{align*}
         \Fcal_h(\varphi u)=\Ocal(h^\infty\langle |\xi|+|\eta|\rangle^{-\infty})
    \end{align*}
    thanks to the wavefront condition ${\rm WF}_h'(P(h))\cap \Delta(S^*X)=\varnothing$.
\end{itemize}
Now \eqref{fouriertrace} gives us 
\begin{align*}
       |\langle\iota^*(\varphi u), \chi \rangle|=h^{-2n}\Ocal(h^{-m})=\Ocal(h^{-2n-m})
\end{align*}
and a partition of unity argument finishes the proof.
\end{proof}

\begin{proof}[Proof of Theorem \ref{traceest}]
The operator $\widetilde{R}_h(z):\mathcal{H}^s_h\to\mathcal{H}^s_h$ is bounded and thus $h$-tempered. Lemma \ref{wf2} gives
\begin{align*}
    \WF_h'(e^{-it_0h^{-1}\widetilde{P}_h(z)}\widetilde{R}_h(z))\cap \Delta(S^*X)=\varnothing
\end{align*}
if we choose $t_0>0$ smaller than the least length of the closed orbits. For any $A,B\in\Psi^{\rm comp}_h(X)$ recall
\begin{align*}
    e^{-it_0P}- e^{-it_0h^{-1}(hP-iQ)}=h^{-1}\int_0^{t_0}e^{-ith^{-1}(hP-iQ)}Qe^{-i(t_0-t)P}dt,
\end{align*}
we have
\begin{align*}
    &\|Ae^{-it_0h^{-1}\widetilde{P}_h(z)} \widetilde{R}_h(z)B\|_{L^2\to L^2}\\
    &\lesssim  \|Ae^{-it_0P} \widetilde{R}_h(z)B\|_{L^2\to L^2}+h^{-1}\int_0^{t_0}\|Ae^{-ith^{-1}(hP-iQ)}Qe^{-i(t_0-t)P}\widetilde{R}_h(z)B\|_{L^2\to L^2}dt\\
    &\lesssim  \|Ae^{-it_0P} \widetilde{R}_h(z)B\|_{\Hcal^s_h\to \Hcal^s_h}+h^{-1}\int_0^{t_0}\|Qe^{-i(t_0-t)P}\widetilde{R}_h(z)B\|_{L^2\to L^2}dt\\
    &=\Ocal(h^{-1})+h^{-1}\int_0^{t_0}\|Qe^{-i(t_0-t)P}\widetilde{R}_h(z)B\|_{\Hcal^s_h\to \Hcal^s_h}dt\\
    &=\Ocal(h^{-2}).
\end{align*}
Here we use the fact that on compact sets in the phase space $L^2$ norm is equivalent to any $\mathcal{H}^s$ norm. Now the claim follows from Lemma \ref{traceest_lem}.
\end{proof}


\subsection*{Acknowledgement}
We would like to thank Semyon Dyatlov for suggesting the argument for estimating compositions of operators in different symbol classes, and Maciej Zworski for numerous discussions and encouraging us to write this note. Long Jin is supported by Recruitment Program of Young
Overseas Talent Plan. Zhongkai Tao gratefully acknowledges partial support under the NSF grant DMS-1901462 and the support of Morningside Center of Mathematics during his visit.



\def\arXiv#1{\href{http://arxiv.org/abs/#1}{arXiv:#1}}

\end{document}